\documentclass[a4paper,10pt]{amsart}
\usepackage[cp1250]{inputenc}
\usepackage[T1]{fontenc} 

\usepackage{times}
\usepackage{amssymb}

\theoremstyle{plain}
\newtheorem{lemma}{Lemma}[section]
\newtheorem{theorem}[lemma]{Theorem}
\newtheorem{corollary}[lemma]{Corollary}

\theoremstyle{remark}
\newtheorem{remark}[lemma]{Remark}

\usepackage{graphicx}
\usepackage{amsmath}
\usepackage{amssymb}
\usepackage{mathrsfs}
\usepackage{stmaryrd}
\usepackage{eucal}
\usepackage{yfonts}
 
\title{Topological expansion of the coefficients of zonal polynomials in genus one.}
\author{Agnieszka Czyżewska-Jankowska}
 
\thanks{The author was supported by a scholarship of European Social Fund}
\thanks{Research partialy supported by Polish Government funds for science, grant N N201 364436 for
the years 2009–2012 }
\begin{document}
\maketitle



\begin{abstract}
We use a combinatorial interpretation of the coefficients of zonal Kerov polynomials as a number of unoriented maps to derive an explicit formula for the coefficients in genus one.
\end{abstract}

\section{Introduction}
Zonal polynomials, introduced by Hua [Hua63, Chapter VI] as a tool in statistics and multivariate analysis quickly became a fundamental tool in this theory and in the random matrix trheory. For an overwiev of the topic we refer to the book of Muirhead [Mui82].

The starting point of this paper is the combinatorial interpretation of coefficients of zonal polynomials (Theorem 1) [FŚ11b]. We will use this interpretation to find an explicit formula for some coefficients, namely coefficients in genus one. 

In order to do that we will study unoriented maps. An unoriented map (or simply a map) is a labeled graph drawn on a surface (compact, 2-dimentional manifold) minimal in the sense that after removing a map we get a collection of open discs. By labeling we mean labeling each side of each of the $n$ edges with different number from $1$ to $2n$.
For full introduction to the topic we adress to the paper of Feray and Śniady [FŚ11b]. In this article we will only consider maps which have one face. Such maps with n edges can be obtained by gluing by pair the edges of a polygon with $2n$ edges. 
More information about it can be found in [FŚ11b].

\begin{theorem}
Consider an edge-labelled polygon of length $2n$. 

Let $s_2,s_3,\dots$ be a sequence of non-negative
integers with only finitely many non-zero elements.
The rescaled coefficient
$$  (-1)^{n+1+2s_2+3s_3+\cdots} (2)^{n-1 - (2s_2+3s_3+\cdots)}
\left[ \left(R_2^{(2)}\right)^{s_2}
\left(R_3^{(2)}\right)^{s_3} \cdots\right]
Z_n $$ of the 
zonal Kerov polynomial is equal to the number of pairs $(M,q)$
such that 
\begin{itemize}
\item $M$ is a connected map obtained by gluing edges of
		our polygon by pair
\item the pair $(G(M),q)$, where $G(M)$ is the underlying graph of $M$,
   fulfill the following conditions:
\begin{itemize}
\item[(a)] the number of the black vertices of $G$ is
equal to $s_2+s_3+\cdots$ 
\item[(b)] the total number of vertices of $G$ is equal to $2 s_2+3 s_3+4 s_4+\cdots$
\item[(c)] $q$ is a function from the set of
the black vertices to the set $\{2,3,\dots\}$; we require that each number
$i\in\{2,3,\dots\}$ is used exactly $s_i$ times;
\item[(d)] \label{enum:marriage_Graphs} for every subset $A\subset V_\circ(G)$ of black vertices of $G$
which is nontrivial (i.e., $A\neq\emptyset$ and $A\neq V_\circ(G)$) there are
more than $\sum_{v\in A} \big( q(v)-1 \big)$ white vertices which are
connected to at least one vertex from $A$

\end{itemize}

\end{itemize}
\end{theorem}

\begin{remark}
It follows directly from the Euler equation (Euler characteristic) that
for a given coefficient of $Z_n$, all the maps from Theorem 1
have the same genus. Thus it makes sense to define the genus of a 
coefficient. Note that it also follows directly from Euler equation that
$2 s_2+3 s_3+4 s_4+\cdots=n-1$ for coefficients in genus one. 
\end{remark}

The following theorem comes from the paper [DFŚ10, Prop. A1].

\begin{theorem}
Condition (d) of Theorem 1 is equivalent
to the following one:
\begin{enumerate}
\item 
it is possible to chose orientations on the edges of the bipartite graph
$G$ in such a way that: 
\begin{itemize}
\item every white vertex has exactly one
outgoing edge and every black vertex $j$ has exactly $q(j)-1$
incoming edges, \item if we interpret orientations of edges as directions of
one-way streets, there exists a closed walk in the graph such that every black
vertex is visited at least once. 
\end{itemize}
\end{enumerate}
\end{theorem}

\section{Results}

Our purpose now is to find a formula for an arbitrary coefficient  of  $R_{a_1}...R_{a_k}$ in $Z_n$ in genus 1.
Our method is to use Theorem 1 so we will need to know the number of
unoriented maps in genus 1 fulfilling some additional conditions, for example they must have one face and they cannot have any disconnecting edges.

We will first show some simple maps in genus 1 which have the minimal possible number of edges and are not yet bipartite graphs. If we ignore the labeling of the edges, we can think that there are 5 basic types of such unoriented maps in genus 1. These maps are shown on figure 1.

\begin{figure}[tb]
\includegraphics[width=100mm]{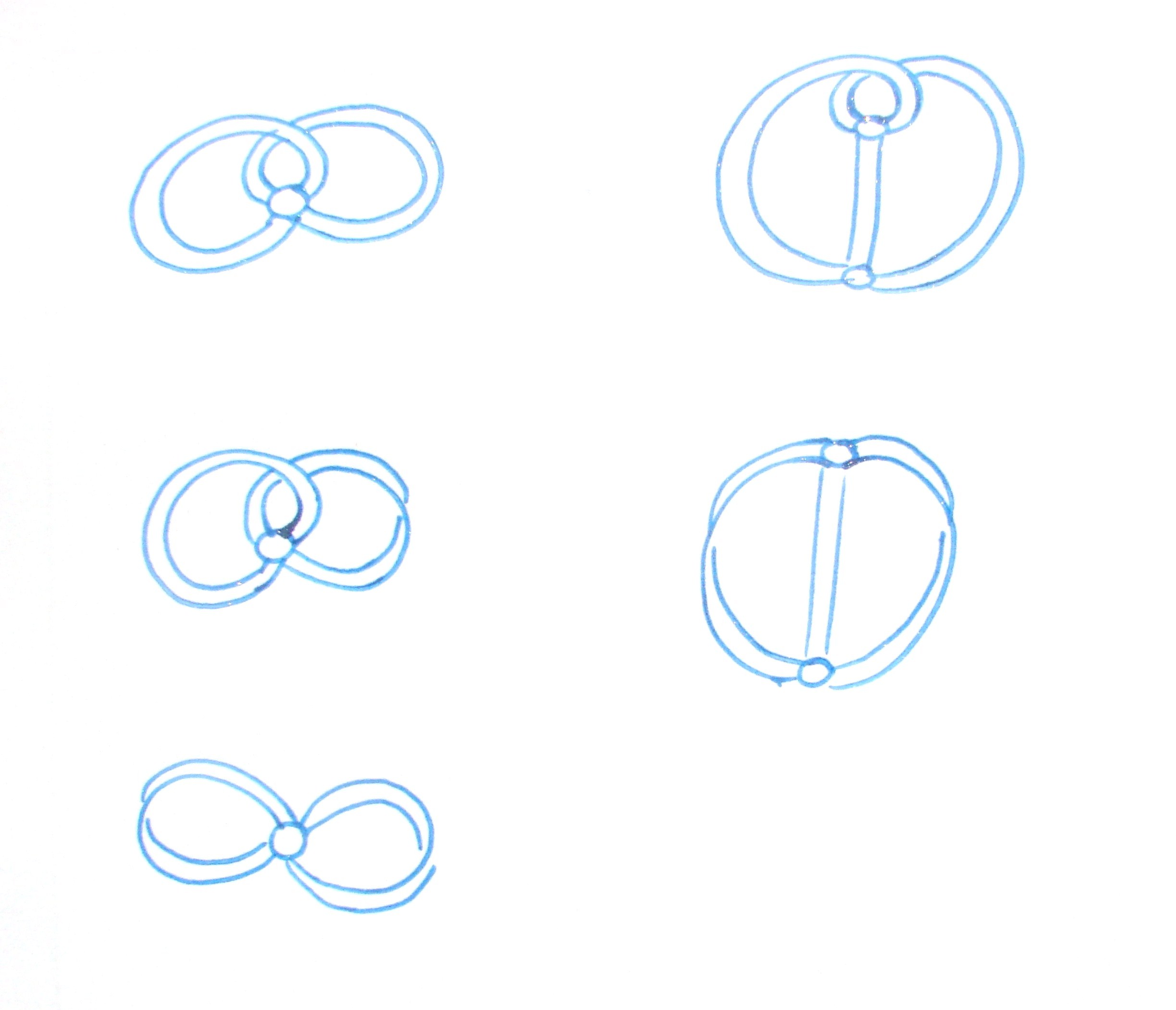}
\caption{Reduced maps in genus 1 with one face without disconnecting edges.}
\end{figure}

Such objects will be called unoriented maps. Note that these basic (reduced) maps do not have vertices of order 1 or 2.
Let us now produce more complicated maps out of the basic ones in the following steps: 
\begin{enumerate}
\item[1.] 
We color existing vertices of the reduced map black and white and add the necessary new vertices of order two to make our map a bipartite graph. We also 
have to think of all possible colorings $q$ of the set of black vertices of our map M.


\item[2.] We add new pairs of black and white vertices on the existing edges. The new vertices will then have color 2.
\item[3.] Around the existing black vertices we add white vertices of order 1. 
We have to think how it changes the coloring $q$ of black vertices.
\item[4.] 
We mark one edge of the map M.
\end{enumerate}

\textbf{1.Bipartite reduced maps with one face, without disconnecting edges. }
First we will see in how many ways we can color black and white the vertices of reduced maps to obtain a bipartite graph.
Most of them need some additional vertices to become a bipartite graph.
Figure 2 shows all (up to the labeling) bipartite reduced maps that we are interested in.
\begin{figure}[tb]
\includegraphics[width=100mm]{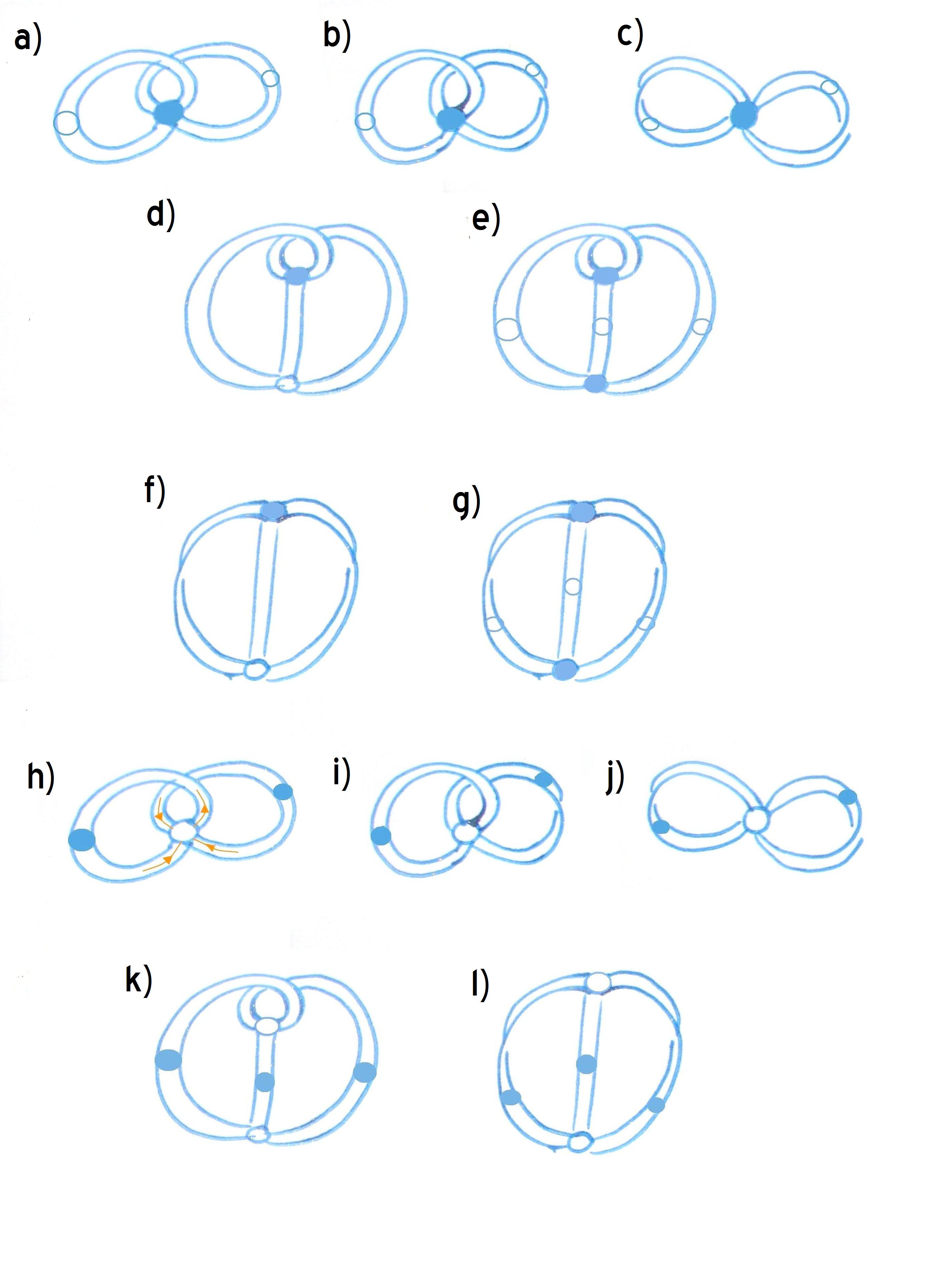}
\caption{Reduced bipartite maps in genus 1}
\end{figure}
\begin{remark}
Only cases a)-g) and maps which can be produced out of them by the above-described procedure have contribution to the coefficients in genus 1.
It comes from the fact that only for these maps there exists a coloring q of the set of black vertices fulfilling the condition d) of Theorem 1.
\end{remark}

\textbf{2.Maps obtained from reduced bipartite maps}\\ 
Let us now deal with more complicated objects produced out of the basic ones.
Any non-reduced map fulfilling the conditions of Theorem 1 can be obtained 
from the reduced map by adding pairs of black and white vertices of order 2 on the existing edges and adding white vertices of order 1 to the existing black vertices of the map M.
Note that we do not need to consider maps which have black vertices of order 1 as they do not fulfill the condition d) of Theorem 1. An example of such a more complicated map with the coloring of black vertices obtained from the case g) is shown on figure 3.
\begin{figure}[tb]
\includegraphics[width=100mm]{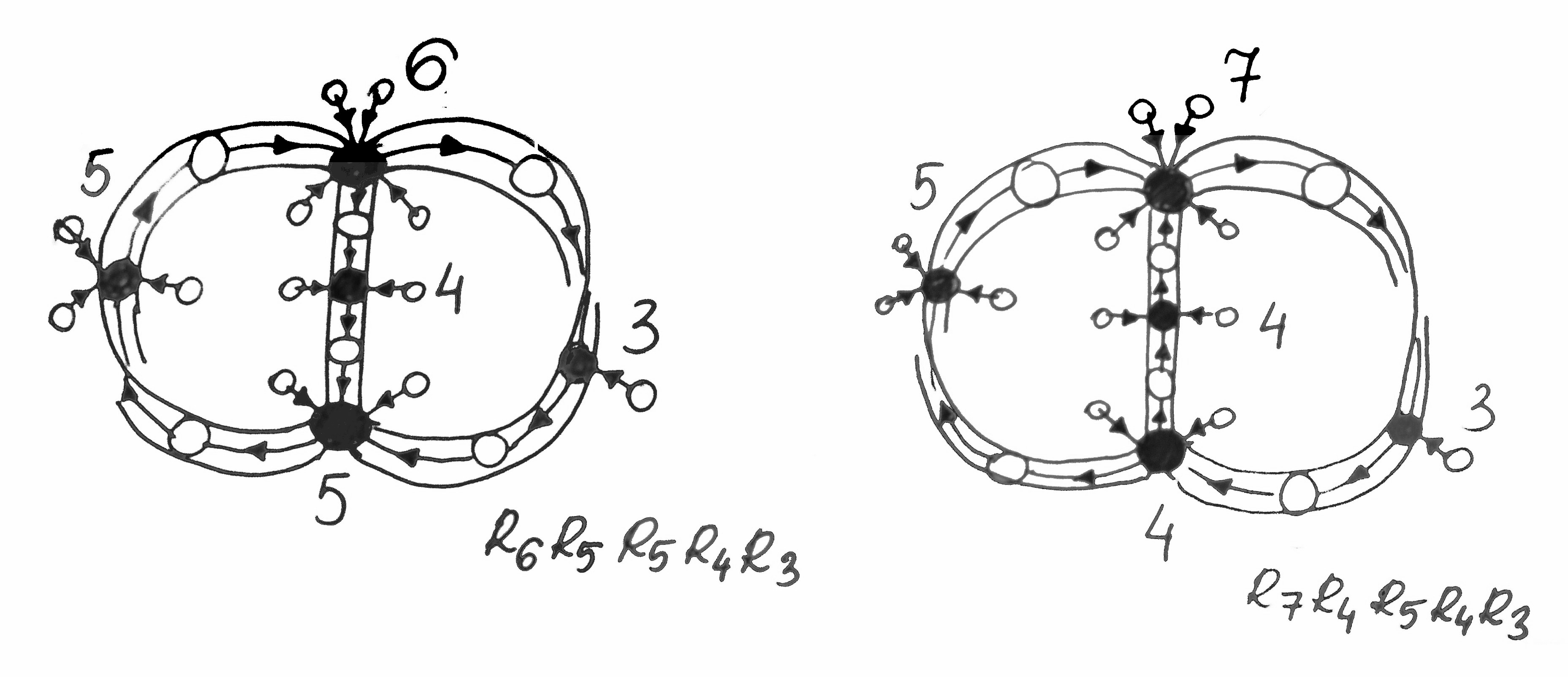}
\caption{Example of a map with the coloring of black vertices.}
\end{figure}

Let us now count the number of ways of adding new vertices of order 1 and 2 to the reduced bipartite map.
\begin{remark}
Let M be a reduced bipartite map with $m$ edges and
let $k$ be the number of black vertices. Then there are
$$
\left({k+m-1}\atop{m-1}\right)=\frac{(k+1)\cdots(k+m-1)}{(m-1)!}
$$
ways of puting $k$ black vertices on $m$ edges of the map M.
The number of ways of adding $k$ white vertices of order 1 to every black vertex of order $m$ is the same.
\end{remark}  

\textbf{3.The coloring of black vertices. }
Let us focus on the number of all possible colorings $q$ of black vertices
of reduced bipartite map and the way this coloring changes when more complicated maps are produced. We start with the most general case and the other cases will
be discussed only briefly. 
Let us take a look at the map g) from figure 2 and see what are the possible colorings $q$ of the black vertices depending on the choice of directions of moving along the edges. As the map g) must have a closed path joining all the black vertices, there are only two possibilities shown on figure 4.
\begin{figure}[tb]
\includegraphics[width=100mm]{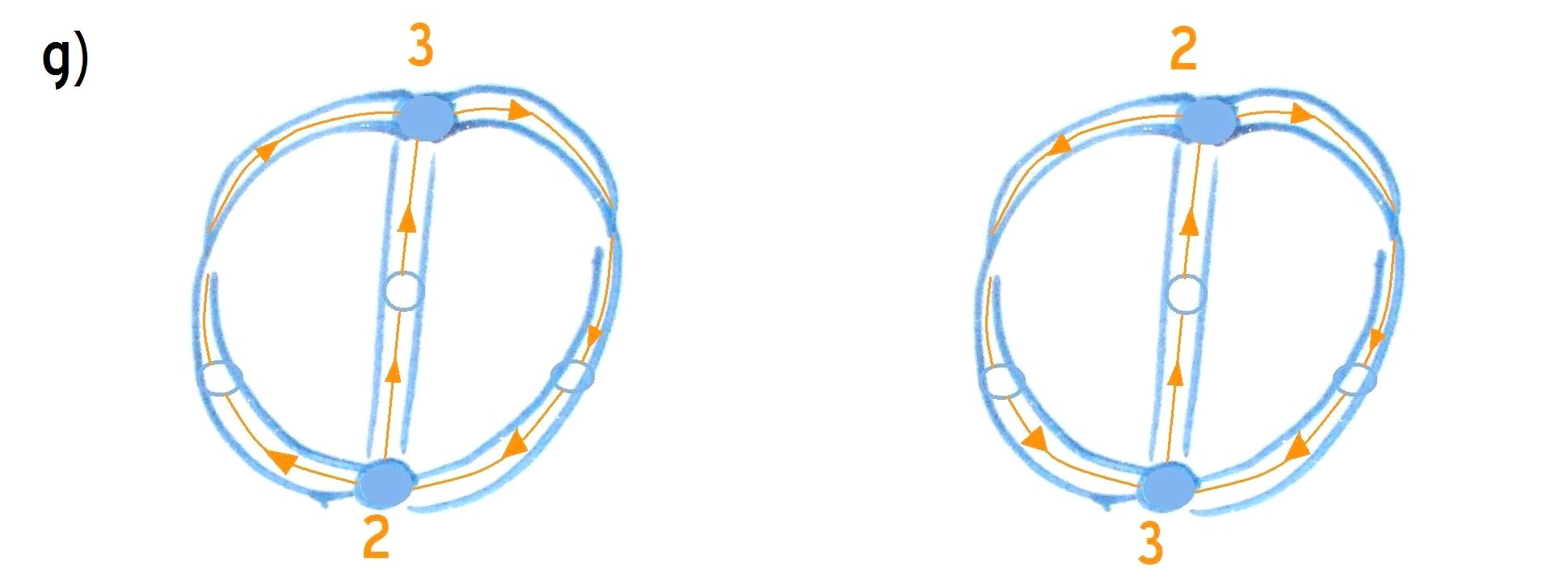}
\caption{Two colorings of a bipartite map from Figure 2 g). }
\end{figure}

Let us see how the coloring of the black vertices of map g) changes when 
we add some new vertices. Let us first add pairs of white and black vertices
of order 2. Note that all the vertices of order 2 must have color 2.


Let us now enumerate black vertices of map M, starting with two vertices of order 3 and enumerate the other vertices along the edges.
Let us then add  $a_1-3$ white vertices of order 1 to the first vertex, $a_2-2$ to the second and $a_i-2$ to the $i$-th vertex of order 2.
The $i$-th black vertex has then color $a_i$ and the resulting map has
contribution to the coefficient of $R_{a_1}...R_{a_k}$.


\begin{figure}[tb]
\includegraphics[width=70mm]{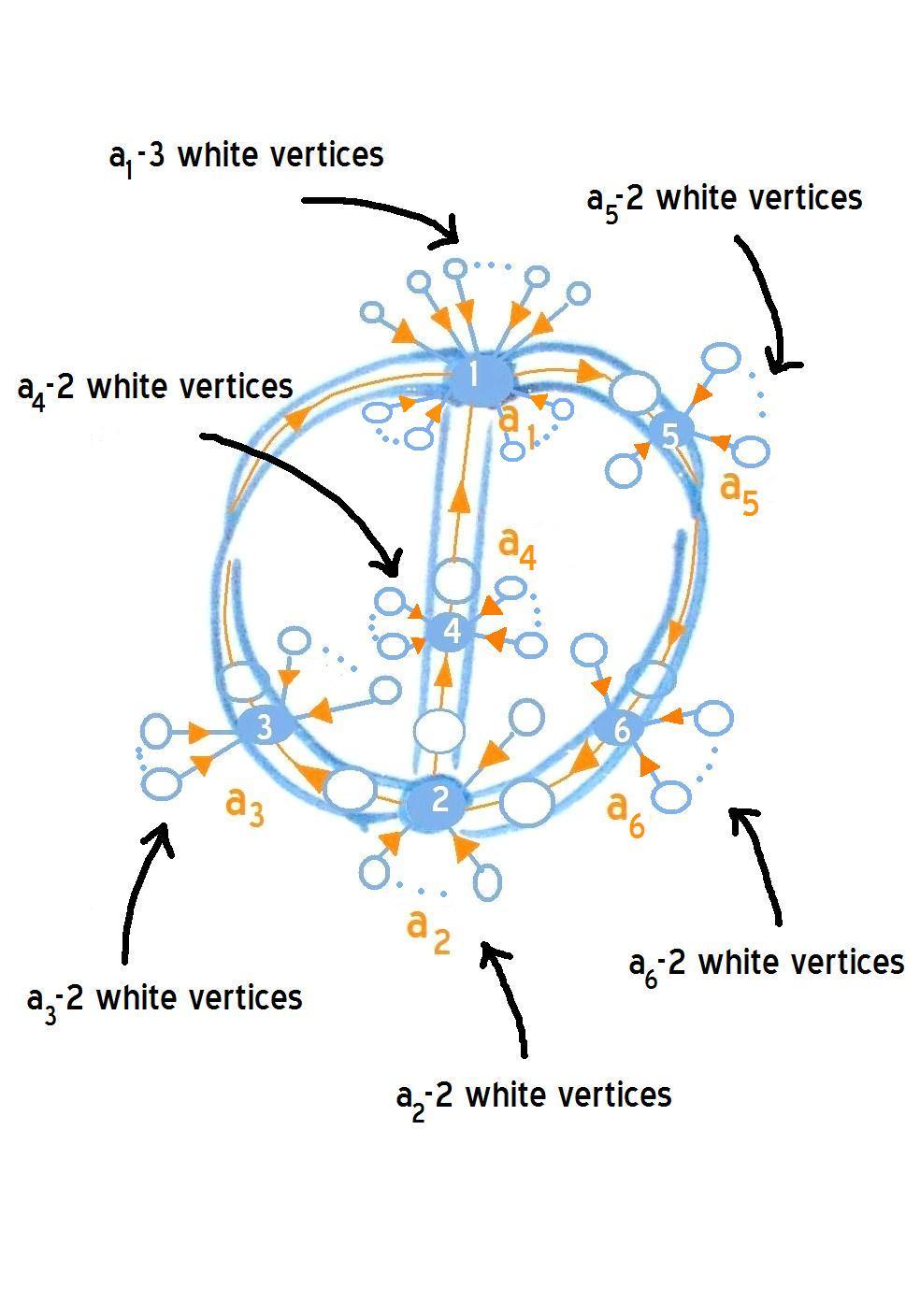}
\caption{This map has contribution to the coefficient of $R_{a_1}...R_{a_k}$ 
and is obtained from the reduced bipartite map g) from Figure 2. The white numbers show the enumeration of black vertices and the numbers $a_i$ is the coloring of the $i$-th black vertex.}
\end{figure}

\begin{remark}
Assume we have a map M without any vertices of order 1 and a coloring of the black vertices. Let $q_i$ be the color of the $i$-th black vertex of M.
If we add $w_i$ white vertices to the $i$-th  black vertex, then it will
change color from $q_i$ to $q_i + w_i$. 
\end{remark}

Thus if we want to build a map which has contribution to the coefficient of
$R_{a_1}...R_{a_k}$ out of the reduced map M with $b$ colored black vertices
and $m$ edges, we can do it using the following procedure:
\begin{enumerate}
\item[a)] Add $k-b$ black vertices of order 2 and give them color 2.
There are $\frac{(k+1)...(k+m-1)}{(m-1)!}$ ways to do it.
\item[b)] Enumerate black vertices in an arbitrary way and add  $a_i-q_i$ white vertices to the $i$-th black vertex. There are $\frac{(a_i-q_i+1)...(a_i-q_i+deg_i-1)}{(deg_i-1)!}$ ways to do that ($deg_i$ is the degree of the $i$-th black vertex).
\end{enumerate}  

It is possible to repeat the above-described procedure for other reduced maps to
get all maps having contribution to the coefficient of $R_{a_1}...R_{a_k}$. To do that we need to know the number of colorings $q$ of black vertices of other reduced bipartite maps. It is ilustrated by figure 6.
Note that it is not necessary to consider all the cases, as the number of colorings of black vertices depends only on the structure of the underlying graph, so it suffices to consider cases a), c) and f). 
Just like in the case~g), knowing the colorings $q$ of black vertices of other reduced bipartite maps, we are able to produce more complicated maps which have contribution to the coefficient of $R_{a_1}...R_{a_k}$.


\begin{figure}[htb]
\includegraphics[width=60mm]{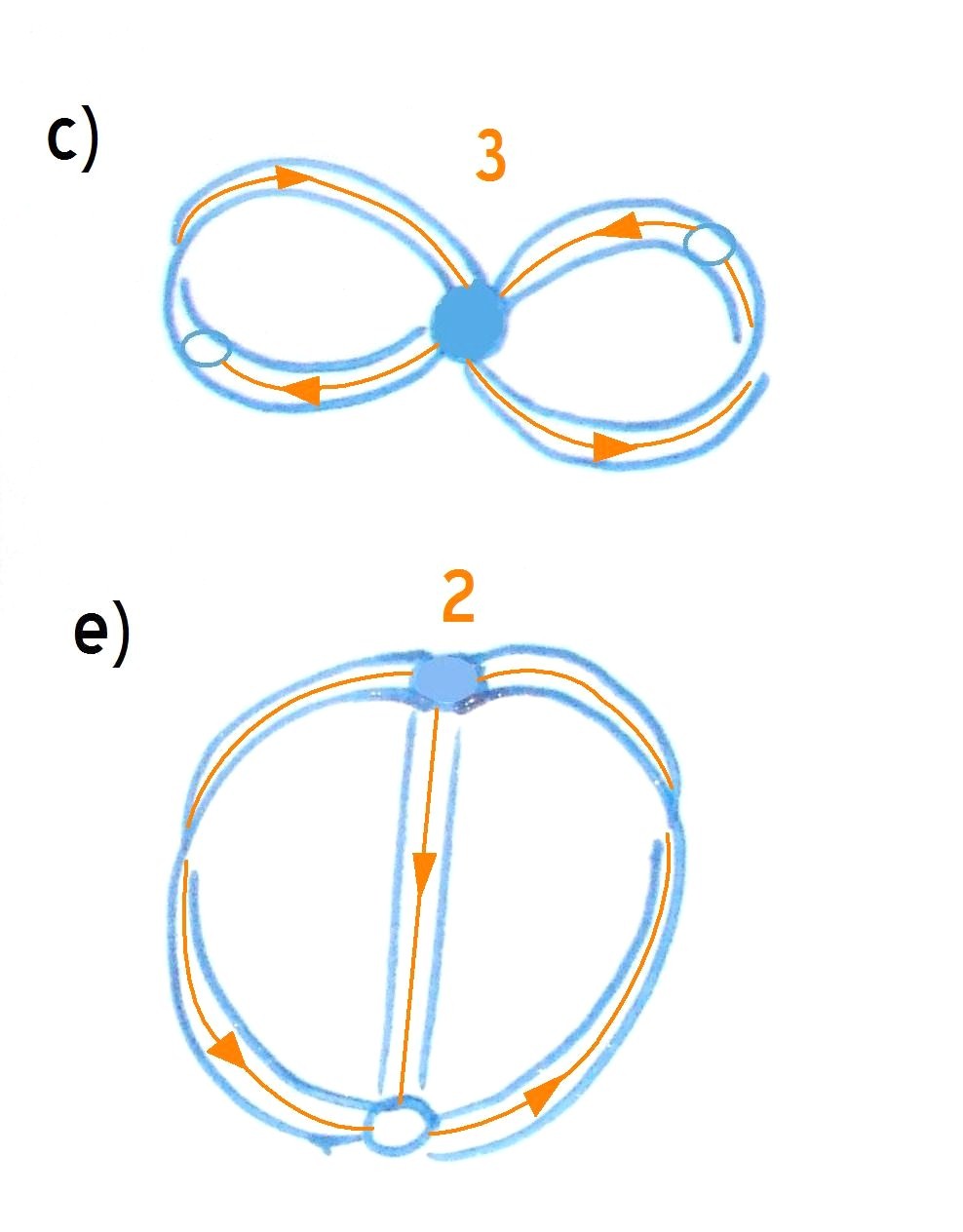}
\caption{Possible colorings of reduced bipartite maps from cases c) and d) from
Figure 2. Other reduced bipartite maps can be colored analogously.}
\end{figure}

\textbf{4. Labeling of the edges.} 
We are now able to find all maps in genus 1 up to the labeling of the edges.
What we need to know is the number of labeled maps.
We will first count the number of labelings of the edges of the reduced bipartite maps.
Let us take a look at the group of symmetries of 
the reduced bipartite maps, showing them as non-glued polygons with marked edges to be glued and one chosen way of labeling the edges.


Figure 7 shows reduced bipartite maps in two equivalent ways with description of
groups of symmetries of a given map and the number of elements of a stabilizator $Stab(M)$ of a given map $M$ under the action of the cyclic group $Z_n$ on the edges of the corresponding polygon.

\begin{figure}[htb]
\includegraphics[width=130mm]{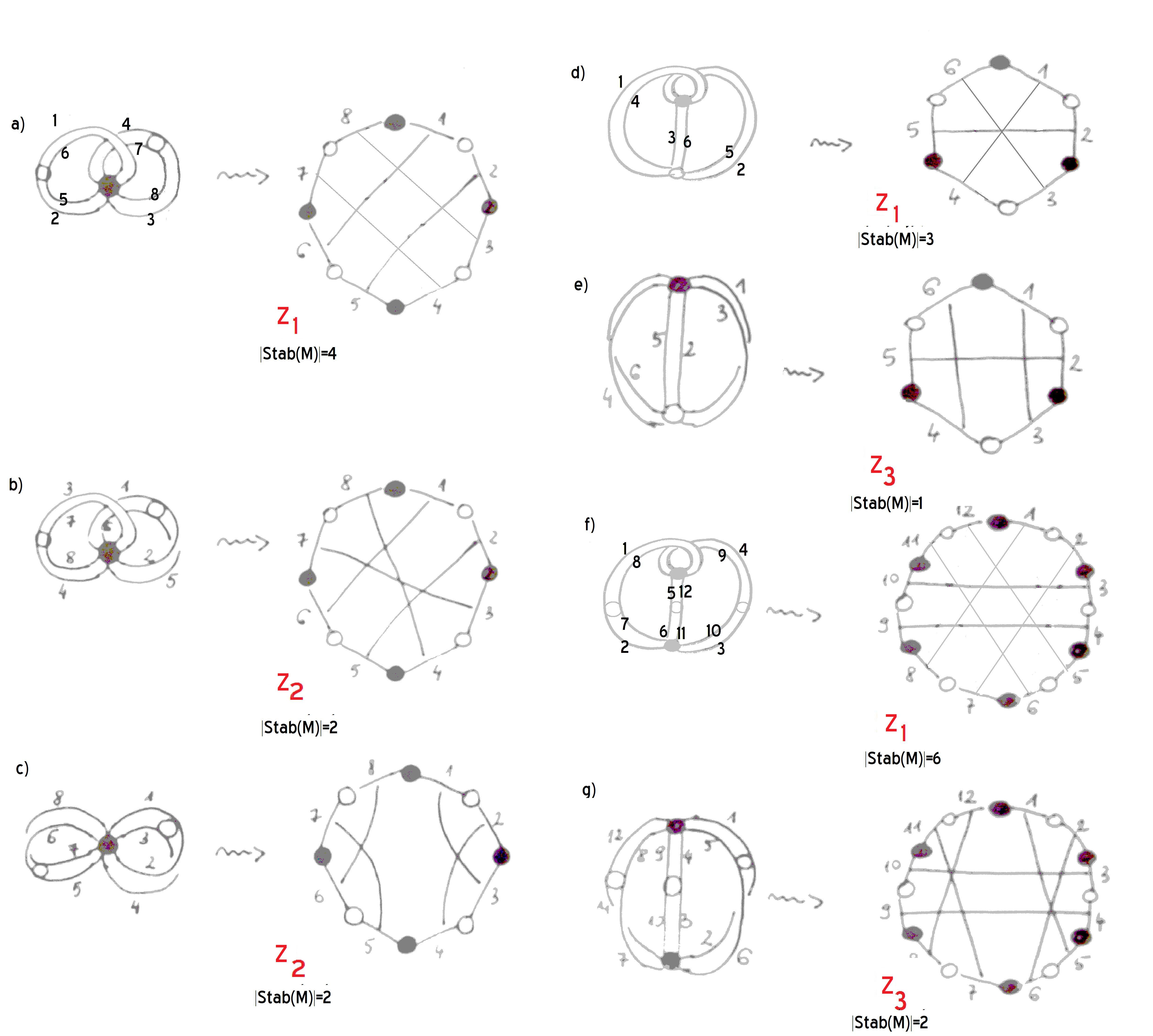}
\caption{Groups of symmetries of reduced bipartite maps.}
\end{figure}

We will now count the number of labeled maps obtained from one reduced map M.
We will use the following procedure:

\begin{enumerate}
\item[a)] Fix an arbitrary reduced map M and choose the labeling of its edges (or mark one edge).
\item[b)] Choose some edges to add vertices of order 2 and then add vertices of order 1 to black vertices in such a way that the resulting map have $b$ black and $w$ white vertices and $n$ edges. Begin labeling of the edges with the same edge we did with the reduced map.
\item[c)] Change the labels of the edges in all possible ways.
\end{enumerate} 
Let M have $k$ edges and let $R_M^{b,w}$ be the number of ways of adding to M vertices of order 1 and 2 in such a way that the resulting map has
 $b$ black and $w$ white vertices.
Note that the above-descried procedure can be performed in  $R_M^{b,w}\cdot n$ ways.


We need to figure out how many times each map is produced in this procedure.
Figure 8 shows an exaple of performing the procedure starting with the case c)
from Figure~2. In this case every map is obtained in two different ways.

\begin{figure}[htb]
\includegraphics[width=120mm]{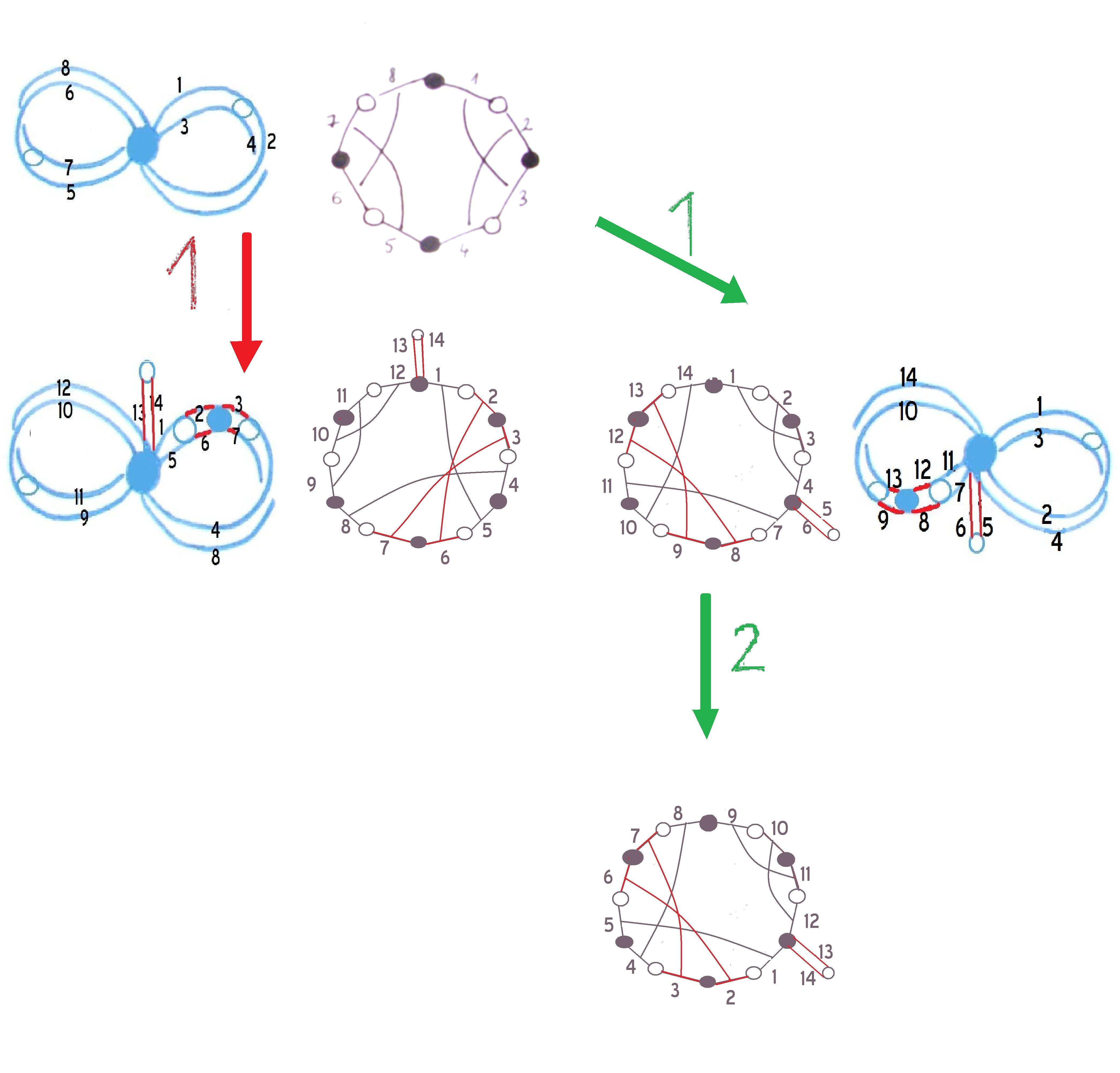}
\caption{Two ways of getting the same map from one reduced map.}
\end{figure}

In general the following is true:
\begin{remark}
The number of maps obtained from a given reduced map M in the above-described procedure is equal $R_M^{b,w}\cdot\frac{n}{Stab(M)}$.
\end{remark} 

\section{ The formula for $Z_{n,n-1}$} 
Now as we know the number of maps in genus 1 we can use Theorem 1 and
formulate the following lemma:
\begin{lemma} The part of the zonal Kerov Polynomial corresponding to the coefficients of genus one is given by the following formula:
\begin{multline*}
Z_{n,n-1}=\sum_k\sum_{{a_1,...,a_k}\atop{a_1+...+a_k=n-1}}\Bigl(
(k\frac{n}{4}+2k\frac{n}{2})\frac{(a_1-2)(a_1-1)a_1}{6}(a_2-1)(a_3-1)...(a_k-1)+\\
(\frac{(k+1)k}{2}\frac{n}{3}+\frac{(k+1)k}{2}n)\frac{(a_1-1)a_1}{2}(a_2-1)(a_3-1)...(a_k-1)+\\
(2\frac{k(k-1)}{2}\frac{n}{6}+2\frac{k(k-1)}{2}\frac{n}{2})\frac{(a_1-1)a_1}{2}\frac{(a_2-2)(a_2-1)}{2}(a_3-1)...(a_k-1)                                                    \Bigr)R_{a_1}...R_{a_k}
\end{multline*}
\end{lemma}
\begin{proof}[Proof]
Every summand in the equation corresponds to one reduced bipartite map.
Let us take a look at one of the summands:

\begin{multline*}
\underbrace{2}_{\text{2\ colorings}}\underbrace{\frac{k(k-1)}{2}}_{\text{putting k-2 black vert. on 3 edges}}
\frac{n}{6}
\underbrace{\frac{(a_1-1)a_1}{2}}_{\text{putting $a_1-2$ white vert. to the black vert. of deg. 3}}\times\\
\times\underbrace{\frac{(a_2-2)(a_2-1)}{2}}_{\text{putting $a_1-3$ white vert. to the other black vert. of deg. 3}}
\underbrace{(a_3-1)...(a_k-1)}_{\text{putting white vert. to the black vert. of deg. 2} }
\end{multline*}
Other summands are obtained in the same way.
\end{proof}
Let us try to show this equation in a more nice form, making it possible to check the Lassalle conjecture in genus 1. 

\begin{theorem}
\begin{multline*}
Z_{n,n-1}=\sum_k\sum_{{a_1,...,a_k}\atop{a_1+...+a_k=n-1}}n\Bigl(\frac{5}{24}\sum_{i=1}^{k} {a_i}^2
+\frac{1}{4}\sum_{i=1}^{k} a_i+
\frac{1}{6}\sum_{{i,j=1}\atop{i\neq j}}^{k} a_ia_j\Bigr)\prod_{i=1}^{i=k}(a_i-1)R_{a_1}\cdots R_{a_k}
\end{multline*}
\end{theorem} 
\begin{proof}[Proof]
Note that in the equation from Lemma 1, instead of  $a_1$ and $a_2$ we can
put arbitrary $a_i$ and the equation will still be true.
Let us then put there  $\sum_{i=1}^{k} a_i$ instead of $ka_1$ and $ka_2$ and then  $\sum_{{i,j=1}\atop{i\neq j}}^{k} a_ia_j$ instead of $k(k-1)a_1a_2$.
The final equation is obtained by simple algebraic operations.
\end{proof}

Now our aim is to read the equation for the coefficient of $R_{a_1}\cdots R_{a_k}$ out of the Theorem 7. Let us see how many times the product $R_{a_1}\cdots R_{a_k}$ appears in our formula. Note that it depends on how many times the numbers $a_i$ appear in the sequence $(a_1,...,a_k)$.
Let  $\mu=(\underbrace{a_1,...,a_1}_{\text{$l_1$ times}},\underbrace{a_2,...,a_2}_{\text{$l_2$ times}},...,\underbrace{a_j,...,a_j}_{\text{$l_j$ times}} )$ for $2\leq a_1 \textless a_2\textless\cdots\textless a_j\in\mathbb{N}$.
Denote by $R_{\mu}$ the product $\underbrace{R_{a_1}\cdots R_{a_1}}_{\text{$l_1$ times}}
\underbrace{R_{a_2}\cdots R_{a_2}}_{\text{$l_2$ times}}\cdots\underbrace{R_{a_j}\cdots R_{a_j}}_{\text{$l_j$ times}}$.
It is easy to see that  $R_{\mu}$ appears in the formula of Theorem 7
$\frac{l(\mu)!}{l_1!\cdot...\cdot l_j!}$ times and the remaining part of the formula does not depend on the order of $a_i$ in the sequence $(a_1,...,a_k)$. We can thus state the following:


\begin{corollary}
Let $\mu=(\underbrace{a_1,...,a_1}_{\text{$l_1$ times}},
\underbrace{a_2,...,a_2}_{\text{$l_2$ times}},...,
\underbrace{a_j,...,a_j}_{\text{$l_j$ times}} )$. Then
\begin{multline*}
Z_{n,n-1}=\sum_k \sum_{|\mu|=n-1}n\frac{l(\mu)!}{l_1!\cdots l_j!}\Bigl(\frac{5}{24}\sum_{i=1}^{k} {a_i}^2
+\frac{1}{4}\sum_{i=1}^{k} a_i+
\frac{1}{6}\sum_{{i,j=1}\atop{i\neq j}}^{k} a_ia_j\Bigr)\prod_{i=1}^{i=k}(a_i-1)R_{\mu}.
\end{multline*}
 
\end{corollary}

\[
\text{REFERENCES}
\]
\begin{enumerate}

\item[[DFŚ10]] Maciej Dołęga, Valentin Féray, and Piotr Śniady. Explicit combinatorial interpretation
of Kerov character polynomials as numbers of permutation factorizations.
Adv. Math., 2010. doi:10.1016/j.aim.2010.02.011.

\item[[FŚ11a]] V. Féray and P. Śniady. Asymptotics of characters of symmetric groups related
to Stanley character formula. Ann. Math., 173(2):887–906, 2011.

\item[[FŚ11b]] V. Féray and P. Śniady. Zonal polynomials via Stanley’s coordinates and free cumulants. J. Alg., 334(1):338–373, 2011.

\item[[Hua63]] L. K. Hua. Harmonic analysis of functions of several complex variables in the classical domains. Translated from the Russian by Leo Ebner and Adam
Korányi. American Mathematical Society, Providence, R.I., 1963.

\item[[Mui82]] Robb J. Muirhead. Aspects of multivariate statistical theory. John Wiley \& Sons Inc., New York, 1982.Wiley Series in Probability and Mathematical Statistics.

\end{enumerate}
\end{document}